\documentclass[11pt, oneside, a4paper]{article}
\usepackage{graphicx}
\usepackage{amsthm}
\usepackage{amsfonts}
\usepackage{amsbsy,amsxtra}
\usepackage{mathtools}
\usepackage{amssymb}
\usepackage{amsmath}
\usepackage{dsfont}
\usepackage{bm}
\usepackage{accents}
\usepackage{enumerate}
\usepackage[utf8]{inputenc} 
\usepackage{physics}
\usepackage{caption}
\usepackage[raggedright]{titlesec}
\usepackage{xcolor}
\usepackage{array}

\usepackage{tikz}
\usepackage{pgfplots}
\pgfplotsset{compat=1.18}
\usetikzlibrary{patterns,patterns.meta,calc}
\usetikzlibrary{decorations.markings, decorations.pathmorphing}
\usetikzlibrary{backgrounds}

\usetikzlibrary{external}
\usepgfplotslibrary{external}
\colorlet{Pcolor}{black!40!white}
\colorlet{PcolorLight}{black!20!white}
\tikzstyle{fill_P_side}=[fill=Pcolor, draw=none]
\tikzstyle{fill_P_side_light}=[fill=PcolorLight, draw=none]

\usepackage[
	english=british
]{csquotes}
\usepackage[
    backend=biber,
    sortlocale=auto,
    natbib=true,	%
    url=false,
    eprint=false,
    maxbibnames=9,
    maxcitenames=2,
    uniquelist=false,	%
    style=authoryear 	%
]{biblatex}

\usepackage{hyperref}
\hypersetup{hypertexnames=false} %
\usepackage{cleveref}
\usepackage{thmtools}   %

\declaretheorem[numberwithin=section,
name=Theorem,
]{theorem}
\declaretheorem[numberlike=theorem,
name=Lemma,
]{lemma}

\declaretheorem[numberlike=theorem,
name=Corollary,
]{corollary}

\declaretheorem[numberlike=theorem,
name=Definition,
style=definition,
]{definition}

\newcommand{\eps}{\ensuremath{\varepsilon}}

\newcommand{\NN}{\ensuremath{\mathds{N}}}  %
\newcommand{\RR}{\ensuremath{\mathds{R}}}  %
\newcommand{\cP}{\ensuremath{\mathcal{P}}} %
\newcommand{\cF}{\ensuremath{\mathcal{F}}} %

\DeclareMathOperator*{\interior}{int}

\DeclareMathOperator{\relu}{ReLU} %

\DeclareMathOperator{\CPA}{CPA}   %
\newcommand{\cl}[1]{\overline{#1}}  %

\DeclarePairedDelimiter{\set}{\lbrace}{\rbrace}

\DeclarePairedDelimiter\floor{\lfloor}{\rfloor}

\newcommand{\titleText}{Bounds on the Number of Pieces in Continuous Piecewise Affine Functions}
\title{\titleText}

\author{ \href{https://orcid.org/0009-0001-9695-3812}{\includegraphics[scale=0.06]{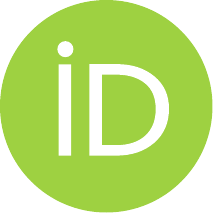}\hspace{1mm}Leo Zanotti}%
		\\
	Institute of Optimization and Operations Research\\
	Ulm University, Germany\\
	\texttt{leo.zanotti@uni-ulm.de}
}

\hypersetup{
	pdftitle={\titleText},
	pdfauthor={Leo Zanotti},
}

\begin{document}
\maketitle

\begin{abstract}
	The complexity of continuous piecewise affine (CPA) functions can be measured by the number of pieces $p$ or the number of distinct affine functions $n$.
	For CPA functions on $\RR^d$, this paper shows an upper bound of $p=O(n^{d+1})$ and constructs a family of functions achieving a lower bound of $p=\Omega(n^{d+1-\frac{c}{\sqrt{\log_2(n)}}})$.
\end{abstract}

\section{Introduction}\label{sec:Introduction}
A continuous piecewise affine (CPA) function, formally defined in \autoref{sec:Defs}, is a continuous function whose domain can be covered by a finite collection of connected sets, such that the function is affine on each set.  
Such functions can be expressed in various analytic forms that involve maxima and minima of affine functions \citep{Tarela1990MaxMin,Koutschan2023,Kripfganz1987}. 
Based on this, \citet{arora2018understanding} showed that every CPA function can be represented by a neural network with the $\relu$ activation function.  
To characterize the expressivity of such neural networks, one seeks bounds on the network size required to represent a CPA function of given complexity \citep{Hertrich2023TowardsLowerBounds,Chen2022neurCompl,Hanin2019}. 
These bounds are based on two different complexity measures for a CPA function $f$: 
the number of pieces $p$, which refers to the sets in its defining cover, and the number of affine components $n$, which counts the distinct affine functions in its piecewise definition.
This naturally raises the question of how these two quantities relate.

Clearly, $p \geq n$, and there exist CPA functions achieving equality. 
However, the maximum possible number of pieces for a given $n$ remains an open question.  
For functions on $\RR^d$, this paper proves an upper bound of $p = O(n^{d+1})$ and shows the existence of CPA functions satisfying a lower bound of $p=\Omega(n^{d+1-\frac{c}{\sqrt{\log_2(n)}}})$. Thus, the upper bound is nearly sharp.

\subsection{Definition of CPA Functions}\label{sec:Defs}

For a subset $P\subset \RR^d$, let $\cl{P}$, $\interior{P}$, and $\partial P$ denote its closure, interior, and boundary, respectively.
For $n\in\NN$, define $[n]:=\set{1,\dots,n}$.
Continuous piecewise affine (CPA) functions are defined as follows.

\begin{definition}\label{def:CPA}
	A continuous function $f: \RR^d \to \RR$ is said to be \emph{CPA} if there exists a finite set $\cP$ of open connected subsets of $\RR^d$ such that $\RR^d = \bigcup_{P\in\cP} \cl{P}$ and $f$ is affine on each $P\in\cP$.
	The sets $P$ are the \emph{pieces} of $f$, and for every $P\in\cP$, the affine function $f|_{P}$ is the \emph{affine component} of $f$ corresponding to $P$.
\end{definition}

Note that the set of pieces is not unique, but we focus on sets with minimal cardinality.
As the proof of \autoref{prop:min_collection_eq_max_pieces} will show, this ensures that the pieces are disjoint and their closures intersect only at their boundaries, i.e., $\cl{P_1}\cap\cl{P_2}=\partial P_1\cap \partial P_2$ for any distinct pieces $P_1,P_2$  
This property is crucial for constructing the lower bound in \autoref{sec:LB}.

\citet{Chen2022neurCompl} define CPA functions differently, requiring the pieces to be closed and connected.  
However, contrary to their claim in Lemma 11(d), this allows cases where $\partial P_1 \cap \interior{P_2} \neq \emptyset$, as shown in \autoref{fig:bad_piece}.

\begin{figure}[h!]
	\centering
	\def\zmax{2.5}
\tikzsetnextfilename{badPiece_top}
\begin{tikzpicture}[very thick]
	\begin{axis}[
		title={top view},
		width=.4\textwidth,
		axis equal image,
		xmin=0, xmax=4,
		ymin=-4,ymax=0,
		xtick={0,...,4},
		ytick={-4,...,0},
		xlabel={$x$},
		ylabel={$y$},
		legend pos=outer north east,
		legend style={cells={align=center}}
		]
		
		\draw[fill=red, draw=none]
		(axis cs: 0,0) -- (axis cs: 1,-1) -- (axis cs: 1,0) -- cycle;
		\draw[fill=red, draw=none]
		(axis cs: 3,0) -- (axis cs: 3,-3) -- (axis cs: 4,-4) -- (axis cs: 4,0) -- cycle;
		\addplot[
		color=red,
		ultra thick,
		forget plot,
		] coordinates {
			(0,0)(1,-1)
		};
		\addplot[
		color=red,
		ultra thick,
		forget plot,
		] coordinates {
			(3,-3)(4,-4)
		};
		\addplot[
		color=red,
		ultra thick,
		dashed,
		forget plot,
		] coordinates {
			(0,0)(4,-4)
		};
		
		\begin{scope}[on background layer]
			\draw[fill=blue, draw=none]
			(axis cs: 1,0) -- (axis cs: 1,-1) -- (axis cs: 2,-2) -- (axis cs: 2,0) -- cycle;
			\draw[fill=green, draw=none]
			(axis cs: 2,0) -- (axis cs: 2,-2) -- (axis cs: 3,-3) -- (axis cs: 3,0) -- cycle;
			
			\draw[fill_P_side]%
			(axis cs: 0,0) -- (axis cs: 1,-1) -- (axis cs: 2,-1) -- (axis cs: 3,-3) -- (axis cs: 4,-4) -- (axis cs: 0,-4) -- cycle;
		\end{scope}

		\addlegendimage{empty legend}
		\addlegendentry{\hspace{-.6cm}affine \\ \hspace{-.6cm}component}
		\addlegendimage{area legend, fill=red}
		\addlegendentry{$-x$}
		
		\addlegendimage{area legend, fill=blue}
		\addlegendentry{$-1$}
		
		\addlegendimage{area legend, fill=green}
		\addlegendentry{$3-2x$}
		
		\addlegendimage{area legend, fill=Pcolor}
		\addlegendentry{$y$}
	\end{axis}
\end{tikzpicture}
	\def\zmax{2.5}
\tikzsetnextfilename{badPiece_slice}
\begin{tikzpicture}[very thick]
	\begin{axis}[
		title={$y=0$ slice},
		width=.4\textwidth,
		axis equal image,
		xmin=0, xmax=4,
		ymin=-4,ymax=0,
		xtick={0,...,4},
		ytick={-4,...,0},
		xlabel={$x$},
		ylabel={$f(x,0)$},
		ylabel near ticks,
		]
		
		\addplot[
		color=blue,
		ultra thick,
		] coordinates {
			(1,-1)(2,-1)
		};
		
		\addplot[
		color=green,
		ultra thick,
		] coordinates {
			(2,-1)(3,-3)
		};
		
		\addplot[
		color=red,
		ultra thick,
		] coordinates {
			(0,0)(1,-1)
		};
		\addplot[
		color=red,
		ultra thick,
		] coordinates {
			(3,-3)(4,-4)
		};

	\end{axis}
\end{tikzpicture}
	\caption{$f(x,y):=\min\big(y,\min\big(\max(-x,-1),\max(3-2x,-x)\big)\big)$. 
		Left: Piecewise definition of $f$.
		Right: Slice of $f$ at $y=0$.
		If the pieces are only required to be closed connected sets, the dashed line $l=\set{(x,y):y=-x}$ can belong to both the red and grey pieces, as $f(x,y)=y=-x$ along this line. 
		Thus, defining $f$ with closed connected sets would require only four pieces, whereas under \autoref{def:CPA}, at least five pieces are necessary.}
	\label{fig:bad_piece}
\end{figure}
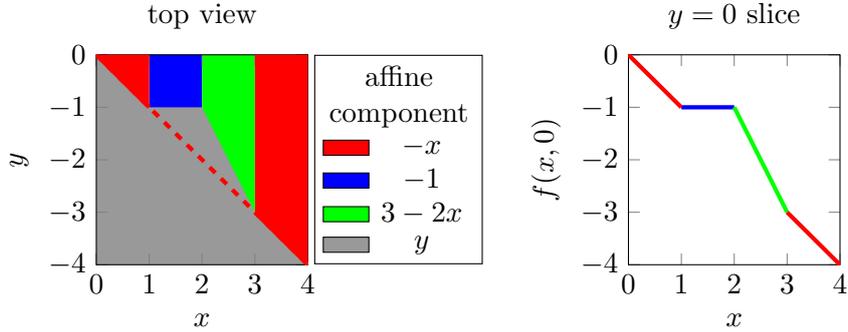

\section{Upper Bound}
The best currently known upper bound on the minimum number of pieces is the following.
\begin{lemma}\label{lem:UB_Chen}
	Let $f:\RR^d\to\RR$ be a CPA function with $n$ distinct affine components. 
	Let $\lambda$ be the minimum number of convex sets needed to satisfy the conditions of \autoref{def:CPA}.
	Then,
	\begin{equation*}
		n\leq \lambda \leq \min\left(\sum_{i=0}^d\binom{\frac{n^2-n}{2}}{i}, n!\right)=O(n^{2d})
	\end{equation*}
\end{lemma}
\begin{proof}
	For closed convex pieces, this would correspond to Lemma 1 in \citet{Chen2022neurCompl}.
	Since the interiors of convex sets are open and connected, the result extends to CPA functions as defined in \autoref{def:CPA}.
\end{proof}

The following theorem shows that the exponent can be improved.
\begin{theorem}\label{thm:many_pieces_UB}
	In the setting of \autoref{lem:UB_Chen}, we have
	\begin{equation*}
		\lambda=O(n^{d+1})
	\end{equation*}
\end{theorem}
\begin{proof}
	Let $\set{f_i}_{i=1}^n$ be the $n$ distinct affine components of $f$. Their graphs
	\begin{equation*}
		H_i := \{(x, f_i(x)) : x \in \mathbb{R}^d\} \subset \mathbb{R}^{d+1},
	\end{equation*}
	define hyperplanes in $\RR^{d+1}$. These hyperplanes are non-vertical, meaning that the projection $\Pi:\RR^{d+1}\to\RR^d$, given by 
	\begin{equation}\label{eq:def:Projection}
		\Pi(x_1,\dots,x_d,y):=(x_1,\dots,x_d),
	\end{equation}
	is injective on each $H_i$.
	For $i\in [n]$, let $\cF_i^\circ$ be defined as \begin{equation*}
		\cF_i^\circ:= \text{the family of connected components of } H_i\setminus \bigcup_{j\neq i} H_j.
	\end{equation*}
	The following subset contains the components that intersect the graph of $f$:
	\begin{equation*}
		\cF_i:=\set{F\in \cF_i^\circ : f(x)=f_i(x) \text{ for some } x\in \Pi F}
	\end{equation*}
	Now, the family of sets \begin{equation*}
		\cP:= \set{\Pi F: i\in [n], F\in\cF_i}
	\end{equation*}	
	defines admissible convex pieces for $f$:\begin{itemize}
		\item In the discrete geometry literature (e.g., \citet{Matousek2002_Discrete_Geometry}), the set $\bigcup_{i\in [n]} \cF_i^\circ$ is known as the set of facets of the arrangement of $\set{H_i}_{i\in[n]}$. 
		Any facet $F\in\cF_i^\circ$ is convex and open in the subspace $H_i$. Since $H_i$ is non-vertical and $\Pi$ is a linear projection, $\Pi F$ is an open convex set.
		\item $f|_P$ is affine for all $P \in \cP$:
		Let $P=\Pi F$ for some $F\in\cF_i$, and $i\in [n]$.
		By the definitions of $\cF_i^\circ\supseteq \cF_i$ and $H_i$, we have $f_i(x)\neq f_j(x)$ for all $x\in\Pi F$ and all $j\neq i$.
		Since $f-f_j$ is continuous for every $j\neq i$, it follows that the set $X:=\set{x\in \Pi F: f(x)=f_i(x)}$ is open.
		
		Moreover, since $f-f_i$ is continuous, we have $f(x)=f_i(x)$ for all $x\in\partial X$. 
		As $X$ is open, we obtain $x \notin \Pi F$ for all $x \in \partial X$. Given that $X \subseteq \Pi F$, this implies $\partial X \subseteq \partial (\Pi F)$.
		
		Furthermore, $X$ is non-empty by the definition of $\cF_i$. Since $\Pi F$ is connected, we conclude that $X = \Pi F$. Indeed, for any $y \in \Pi F$ and any $x \in X$, there exists a path $\gamma_{xy} \subset \Pi F$ connecting $y$ to $x$. In particular, $\gamma_{xy}\cap \partial X=\emptyset$, which ensures that $y \in X$.
		
		Altogether, we get that $f(x)=f_i(x)$ for all $x\in\Pi F$.
		\item $Q:=\bigcup_{P\in \cP}\cl{P}$ covers $\RR^d$:
		Let $x\in\RR^d$ be arbitrary. Since all the $f_i$ are distinct, the sets \begin{equation*}
			H_{ij}:=\set{x\in\RR^d:f_i(x)= f_j(x)},\quad i\neq j,
		\end{equation*}
		are affine subspaces of dimension $d-1$ or empty.
		Let $\eps>0$ be arbitrary, and consider the ball $B_\eps(x)$ of radius $\eps$ centered at $x$. We obtain \begin{equation*}
			B_\eps(x)\setminus\bigcup_{i\neq j} H_{ij}\neq \emptyset.
		\end{equation*}
		This implies that there exists $x_\eps\in B_\eps(x)$ and some $i_\eps\in[n]$ such that \begin{equation*}
			f(x_\eps)=f_{i_\eps}(x_\eps)\neq f_i(x_\eps)\quad\forall i\neq i_\eps.
		\end{equation*}
		Thus, there exists $F\in \cF_{i_\eps}^\circ$ with $x_\eps\in \Pi F$. In particular, $F\in \cF_{i_\eps}$, so $x_\eps\in \Pi F\in \cP$.
		Since $\eps>0$ was arbitrary and $Q$ is closed, we conclude that $x\in Q$.
	\end{itemize}
	It is well known (see, e.g., \citet{Matousek2002_Discrete_Geometry}) that the number of facets is bounded by
	\begin{equation*}
		\sum_{i=1}^n\abs{\cF_i^\circ}\leq \sum_{i=1}^n\sum_{k=0}^{d}\binom{n-1}{k}\leq\sum_{i=1}^n (d+1)n^d=O(n^{d+1}).
	\end{equation*}
	Therefore, we conclude that $\abs{\cP}=O(n^{d+1})$.
\end{proof}

\section{Lower Bound}\label{sec:LB}
Rather than considering a minimum number of pieces, this section focuses on maximal pieces.
\begin{definition}
	Let $f:\RR^d\to\RR$ be a CPA function with pieces $\cP$. 
	A piece $P\in\cP$ is called \emph{maximal} if there does not exist a nonempty set $U\subseteq\RR^d\setminus P$ such that $P\cup U$ is open and connected, and $f|_{P\cup U}$ remains affine.
\end{definition}
The following lemma establishes the equivalence between a minimal set of pieces and maximal pieces.
\begin{lemma}\label{prop:min_collection_eq_max_pieces}
	For any $\CPA$ function $f$, a set of pieces $\cP$ satisfying \autoref{def:CPA} has minimum cardinality if and only if every $P\in\cP$ is maximal.
\end{lemma}
\begin{proof}
	Let $\cP^\ast$ be a set of pieces with minimum cardinality.
	\begin{itemize}	
		\item[(i)] Suppose there exists some $P\in\cP^\ast$ that is not maximal.  
		By the definition of maximality, there exists a nonempty set $U\subseteq\RR^d\setminus P$ such that $P\cup U$ is open and connected, and $f|_{P\cup U}$ remains affine. 
		Then, $\interior U\neq \emptyset$, and since $\bigcup_{P\in\cP^\ast}\cl{P}=\RR^d$, there exists some $P'\in\cP^\ast\setminus\set{P}$ such that $P'\cap\interior U\neq\emptyset$.  
		As $P'\cap\interior U$ is open, $f$ must be affine on $\tilde{P}:=P\cup U\cup P'$. Moreover, $\tilde{P}$ is open and connected, and we conclude that $\tilde{\cP}^\ast:=\cP^\ast\setminus\set{P,P'}\cup\set{\tilde{P}}$ also satisfies \autoref{def:CPA}, but with $|\tilde{\cP}^\ast|<|\cP^\ast|$, contradicting the minimality of $\cP^\ast$.  

		\item[(ii)] Let $\cP$ be some set of maximal pieces.
		Consider an arbitrary nonempty open connected set $P'$ such that $f|_{P'}$ is affine.
		Since $\bigcup_{P\in\cP}\cl{P}=\RR^d$, there exists some $P\in\cP$ such that $P'\cap P\neq\emptyset$.
		Again, this implies that $f$ is affine on the open connected set $P'\cup P$. As $P$ is maximal, we have $P'\subseteq P$.
		In particular, since all pieces in $\cP^\ast$ are also maximal by (i), for every $P'\in\cP^\ast$, there exists some $P\in\cP$ such that $P'=P$, i.e., $\cP^\ast\subseteq\cP$. 
		Furthermore, the same argument shows that for any $P_1,P_2\in\cP$, $P_1\cap P_2\neq\emptyset$ if and only if $P_1=P_2$.
		
		If we assume $|\cP^\ast|<|\cP|$, then there must exist $P\in\cP\setminus\cP^\ast$ such that $P\cap P'=\emptyset$ for all $P'\in\cP^\ast$.  
		This contradicts the fact that $P$ is open and that $\bigcup_{P'\in\cP^\ast}\cl{P'}=\RR^d$.
	\end{itemize}
\end{proof}

We will prove a lower bound on the number of maximal pieces by induction. 
The base case follows from a result of \citet{balogh2003} on the length of monotone paths in line arrangements in the plane. These monotone paths are defined as follows.
\begin{definition}
	For a given set of lines $L=\set{l_1,\dots,l_n}$ in $\RR^2$, a path is a simple polygonal chain that joins a set of distinct vertices in $V=\set{l_i\cap l_j, i<j}$ with segments contained in lines of $L$. The length of a path is the minimum number of segments required.
	A path is monotone if there is a line such that the orthogonal projection of the path onto the line is injective.
\end{definition}
\begin{theorem}[{\citet[Corollary 1]{balogh2003}}]\label{thm:balogh_path}
	There exist $\beta,c>0$ such that for every $n\in\NN$, there exists a set of $n$ lines in the plane having a monotone path of length at least $\beta\cdot n^{2-\frac{c}{\sqrt{\log_2(n)}}}$.
\end{theorem}
The following corollary translates this result into the language of 
$\CPA$ functions.
\begin{corollary}\label{cor:Balogh}
	There exist $\beta,c>0$ such that for every $n\in\NN$, there exists a $\CPA$ function $f:\RR\to\RR$, with $n$ affine components and at least $\beta \cdot n^{2-\frac{c}{\sqrt{\log_2(n)}}}$ maximal pieces.
\end{corollary}
\begin{proof}
	Consider the set of lines given by \autoref{thm:balogh_path}, and define the $x$-axis as the line of monotonicity of the path.
	Each line that contains a segment of the path cannot be orthogonal to the $x$-axis.  
	Thus, it can be expressed as the graph of an affine function with respect to $x$.  	
	The entire path then forms the graph of a CPA function whose pieces correspond to the orthogonal projections of the segments onto the $x$-axis.  
	Moreover, the leftmost and rightmost pieces are extended to $-\infty$ and $\infty$, respectively.
	The number of maximal pieces equals the length of the path.
\end{proof}	

\begin{figure}[h!]
	\centering
	\def\zmax{2.5}
\tikzsetnextfilename{sawTooth}
\begin{tikzpicture}[very thick]
	\begin{axis}[
		width=.6\textwidth,
		height=.3\textwidth,
		xmin=-.3, xmax=6.3,
		ymin=-.3*\zmax,%
		xtick={0,...,6},
		xticklabels={0,...,5,$6=2m$},
		ytick={0,\zmax},
		yticklabels={$z_{\min}$, $z_{\max}$},
		legend pos=north west,
		ymajorgrids=true,
		grid style=dashed,
		]
		
		\addplot[
		color=blue,
		very thick,
		]
		coordinates {
			(-.3,-.3*\zmax)(0,0)(1,\zmax)(2,0)(3,\zmax)(4,0)(5,\zmax)(6,0)(6.3,-.3*\zmax)
		};

		\coordinate (TI) at (axis cs: 0,2);		%
		\coordinate (TM) at (axis cs: 0,6);		%
		\coordinate (TO) at (axis cs: 0,10);		%
		\coordinate (TR) at (axis cs: 0,12);		%

	\end{axis}
\end{tikzpicture}                   
	\caption{$3$-sawtooth function.}
	\label{fig:saw_tooth}
\end{figure}
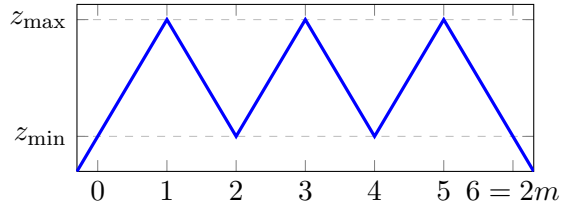

The functions defined next help to replicate the pieces of a CPA function along an additional dimension.

\begin{definition}
	For $z_{\min},z_{\max}\in\RR$ and $m\in\NN$, the $m$-sawtooth function between $z_{\min}$ and $z_{\max}$ is the linear spline interpolating the points $\set{(i,z_i)}_{i=0}^{2m}$, where $z_i=z_{\min}$ if $i$ is even, and $z_i=z_{\max}$ otherwise.
	Outside the interval $[0,2m]$, the function is extended linearly, as illustrated in \autoref{fig:saw_tooth} for the case $m=3$.
\end{definition}

\begin{lemma}\label{lem:LB_counting_lemma}
	Let $f:\RR^d\to\RR$ be a $\CPA$ function with $p$ maximal pieces.
	Suppose that $z_{\min}<\inf f\leq \sup f < z_{\max}$ for some $z_{\min},z_{\max}\in\RR$. Let $s:\RR\to\RR$ be the $m$-sawtooth function between $z_{\min}$ and $z_{\max}$. Then, \begin{equation*}
		h:(x,t)\in\RR^{d+1}\mapsto\min(f(x),s(t))
	\end{equation*}
	is a $\CPA$ function with at least $m\cdot p$ maximal pieces.
\end{lemma}
\begin{proof}
	Let $\set{X_k}_{k=1}^p$ be the $p$ maximal pieces of $f$ with corresponding affine components $f_k$.
	Obviously $h\in\CPA$ with some set of maximal pieces $\cP$.
	
	For every $i\in\set{1,3,5,\dots,2m-1}$ and $k\in [p]$, the pair $(k,i)$ identifies a maximal piece of $h$ in the following way:
	Since $i$ is odd, $s(i)=z_{\max}>f(x)+\eps$ for all $x\in X_k$, and some $\eps>0$. Continuity of $s$ yields an open interval $I\ni i$, such that $s(t)>f(x)$ for all $(x,t)\in X_k\times I$.
	By definition of $h$, we have $h(x,t)=f(x)=f_k(x)$ for all $(x,t)\in X_k\times I$.
	As $X_k\times I$ satisfies \autoref{def:CPA}, there must be a maximal piece $P_{ki}\in\cP$ containing $X_k\times I\subseteq P_{ki}$.
	
	For two distinct pairs $(k,i)\neq(l,j)$, we have $P_{ki}\neq P_{lj}$, as the following arguments show. Assume that $P_{ki}= P_{lj}=:P$, and therefore $h(x,t)=f_k(x)$ for all $(x,t)\in P$.
	Let $x_k\in X_k$ and $x_l\in X_l$ be arbitrary. We know that $(x_k,i),(x_l,j)\in P$ by definition, and since pieces are connected, these two points are connected by a path $\gamma\subset P$.
	\begin{itemize}
		\item case $i\neq j$: Assume that $i<j$. Then, since $i,j$ are odd, also $i<i+1<j$.
		Therefore, there is $x\in\RR^d$ such that $(x,i+1)\in\gamma$.
		Since $i+1$ is even, we get \begin{equation*}
			h(x,i+1)=\min(f(x),s(i+1))=\min(f(x),z_{\min})= z_{\min} < f_k(x),
		\end{equation*}
		which contradicts $(x,i+1)\in\gamma\subset P$.
		\item case $k\neq l$ (cf. \autoref{fig:lifted_function}):
		Using the linear projection $\Pi$ defined in \eqref{eq:def:Projection}, the set $\Pi P\supsetneq X_k$ is open and connected. Since $X_k$ is maximal, there exists $x_0\in \Pi P$ such that $f(x_0)\neq f_k(x_0)$.
		By the definition of $\Pi$, there exists $t_0$ such that $(x_0,t_0)\in P$.
		Since $h(x_0,t_0)=\min(f(x_0),s(t_0))=f_k(x_0)$, we must have $s(t_0) < f(x_0)$.
		By the continuity of $s$, there exists an interval $J \ni t_0$ such that $h(x_0, t) = s(t)$ for all $t \in J$.  
		In particular, $s(t)=h(x_0,t)=f_k(x_0)$ for all $t\in J$ with $(x_0,t)\in P$, which is a contradiction to $f_k(x_0)$ being constant while $s$ is not constant on any interval.
	\end{itemize}
	In total, we can associate to every pair $(k,i)\in [p]\times \set{1,3,\dots,2m-1}$ a maximal piece $P_{ki}\in\cP$ and all the $P_{ki}$ are distinct. Thus $|\cP|\geq m\cdot p$.
\end{proof}

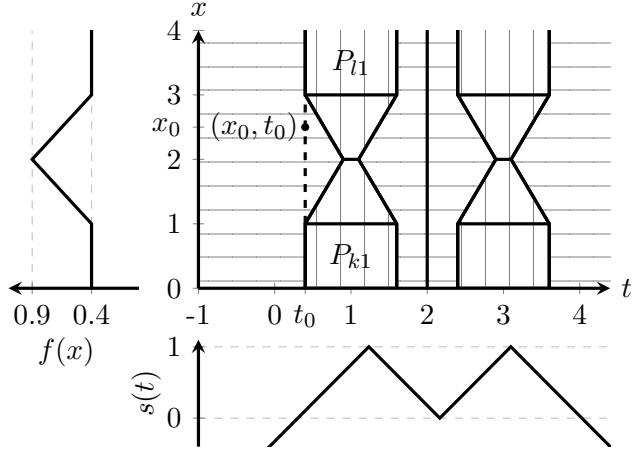
\begin{figure}[h!]
	\centering
	\tikzsetnextfilename{lifted_function}
\begin{tikzpicture}[very thick]
	\begin{axis}[
		width=7cm,
		height=5cm,
		xmin=-1, xmax=4.4,
		ymin=0,ymax=4,
		xtick={-1,0,.4,1,2,3,4},
		xticklabels={-1,0,$t_0$,1,2,3,4},
		ytick={0,1,2,2.5,3,4},
		yticklabels={0,1,2,$x_0$,3,4},
		axis x line=bottom,
		axis y line=left,
		xlabel={$t$},
		ylabel={$x$},
		xlabel style={at={(axis description cs: 1,0)}, anchor=west},
		ylabel style={at={(axis description cs: 0,1.15)}, anchor=south, rotate=90},
		very thick,
		]
		
		\begin{scope}[on background layer]
			\draw[pattern={Lines[angle=0,distance=9pt,line width=.3pt]},pattern color=gray,draw=none] 
			(axis cs: -1,0) -- (axis cs: 4.4,0) -- (axis cs: 4.4,4) -- (axis cs: -1,4) -- cycle;
		\end{scope}   
		
		\pgfplotsinvokeforeach{0,2} {
			\draw[fill=white,postaction={draw,pattern={Lines[angle=90,distance=9pt,line width=.3pt]}},pattern color=gray] 
				(axis cs: #1+0.4,0) -- (axis cs: #1+1.6,0) -- (axis cs: #1+1.6,1) -- (axis cs: #1+1.1,2) -- (axis cs: #1+1.6,3) -- (axis cs: #1+1.6,4.4) -- (axis cs: #1+0.4,4.4) -- (axis cs: #1+0.4,3) -- (axis cs: #1+0.9,2) -- (axis cs: #1+0.4,1) -- cycle;
			\draw (axis cs: #1+1.1,2) -- (axis cs: #1+0.9,2);
			\draw (axis cs: #1+0.4,1) -- (axis cs: #1+1.6,1);
			\draw (axis cs: #1+0.4,3) -- (axis cs: #1+1.6,3);
		}
		\draw (axis cs: 2,-100) -- (axis cs: 2,100);
		
		\draw[dashed] (axis cs: .4,1) -- (axis cs: .4,3);
		
		\fill[] (axis cs: .4,2.5) circle (1.5pt) node[left,inner sep=2pt] {$(x_0,t_0)$};
		
		\node[] at (axis cs: 1,3.5) {$P_{l1}$};
		\node[] at (axis cs: 1,0.5) {$P_{k1}$};
	\end{axis}

	\begin{axis}[
		width=3.3cm,
		height=5cm,
		xmin=0, xmax=1.1,
		ymin=0,ymax=4,
		xtick={0.4,0.9},
		axis x line=bottom,
		axis y line=none,
		x dir=reverse,
		xlabel={$f(x)$},
		xlabel style={xshift=-0.1cm,yshift=0.1cm},
		xmajorgrids=true,
		grid style=dashed,
		very thick,
		at={(-2.5cm,0)}
		]
		
		\draw (axis cs: 0.4,-10) -- (axis cs: 0.4,1) -- (axis cs: 0.9,2) -- (axis cs: 0.4,3) -- (axis cs: 0.4,7);
	\end{axis}
	
	\begin{axis}[
		width=7cm,
		height=3cm,
		xmin=-1.4, xmax=4.4,
		ymin=-.4,ymax=1.1,
		ytick={0,1},
		axis x line=none,
		axis y line=left,
		ylabel={$s(t)$},
		ylabel style={yshift=-0.2cm},
		ymajorgrids=true,
		grid style=dashed,
		very thick,
		at={(0,-2.1cm)},
		]
		
		\draw (axis cs: -100,-100) -- (axis cs: 1,1) -- (axis cs: 2,0) -- (axis cs: 3,1) -- (axis cs: 104,-100);
	\end{axis}
\end{tikzpicture}
	\caption{The middle figure shows the pieces of $h(x,t)=\min(f(x),s(t))$, where $f(x)$ and $s(t)$ are plotted along the left and bottom axes, respectively. 
			$h(x,t)=s(t)$ in horizontally patterned pieces and $h(x,t)=f(x)$ in vertically patterned pieces. 
			For the dashed segment $L$, we have $h|_{L}=h|_{P_{k1}}=h|_{P_{l1}}=0.4$. However, no $(x_0,t_0)\in L$ has a neighbourhood $U$ where $h|_U=0.4$, implying that $(x_0,t_0)$ is not contained in any piece with affine component $0.4$.}
	\label{fig:lifted_function}
\end{figure}

\begin{lemma}\label{cor:lift}
	Let $d\in\NN$.
	Assume that $\set{f_n}_{n\in\NN}$ is a family of $\CPA$ functions $f_n:\RR^d\to\RR$ consisting of at most $\alpha\cdot n$ affine components and having at least $\beta\cdot n^r$ maximal pieces for some constants $\alpha,\beta>0$.
	Then, there exists also a family $\set{h_n}_{n\in\NN}$ of $\CPA$ functions $h_n:\RR^{d+1}\to\RR$ consisting of at most $(\alpha+3) \cdot n$ affine components and having at least $\beta\cdot n^{r+1}$ maximal pieces.
\end{lemma}
\begin{proof}
	Let $n\in\NN$ be fixed.
	By assumption, there exists a ball $B\subset \RR^d$ such that $B\cap P\neq\emptyset$ for at least $\beta n^r$ different maximal pieces $P$ of $f_n$. 
	Continuity of $f_n$ gives $z_{\min},z_{\max}\in\RR$ such that $f_n(B)\subseteq (z_{\min},z_{\max})$. We define $\bar{f}_n(x):=\min(z_{\max},\max(z_{\min},f_n(x)))$, which has at most $\alpha\cdot n+2$ affine components and at least $\beta n^r$ pieces.
	In particular $z_{\min}\leq \bar{f}_n(x)\leq z_{\max}$ for all $x\in\RR$.
	Applying \autoref{lem:LB_counting_lemma} with $m=n$ to $\bar{f}_n$, gives $h_n:\RR^{d+1}\to\RR$ with at least $\beta n^{r+1}$ maximal pieces. 
	Each affine component of $h_n$ is either of the form $(x,t)\mapsto s_0(t)$ for some affine component $s_0$ of $s$, or of the form $(x,t)\mapsto f_0(x)$ for some affine component $f_0$ of $\bar{f}_n$. 
	Thus, $h_n$ has at most $\alpha\cdot n+2+2n\leq (\alpha+3)\cdot n$ affine components.
\end{proof}
Using this result, the construction from \autoref{cor:Balogh} can be lifted into higher dimensions as follows.
\begin{lemma}\label{cor:Balogh_dDim}
	For every $d\in\NN$, there exist $\alpha,\beta,c>0$ such that for every $m\in\NN$, there is a $\CPA$ function $f_m:\RR^d\to\RR$ with at most $\alpha\cdot m$ affine components and at least $\beta \cdot m^{d+1-\frac{c}{\sqrt[]{\log_2(m)}}}$ maximal pieces.
\end{lemma}
\begin{proof}
	For $d=1$, this is precisely \autoref{cor:Balogh}. Applying \autoref{cor:lift} yields the claim by induction.
\end{proof}

The main result of this section is the following theorem, which provides a more explicit relationship between the number of pieces and the number of affine components.

\begin{theorem}\label{thm:many_pieces_LB}
	For every $d\in\NN$, there exist $\beta,c>0$ such that, for every $n\in\NN$, there is a $\CPA$ function $f:\RR^d\to\RR$ with at most $n$ affine components and at least $\beta \cdot n^{d+1-\frac{c}{\sqrt{\log_2(n)}}}$ maximal pieces.
\end{theorem}
\begin{proof}
	Fix $d\in\NN$, and let $\alpha$, $\beta_0$, $c_0$ be the constants given by \autoref{cor:Balogh_dDim}.
	For $n\in\NN$, define $m:=\floor{\frac{n}{\alpha}}$.
	Then, \autoref{cor:Balogh_dDim} provides a function $f$ with at most $\alpha m\leq n$ affine components and at least $\beta_0 m^{d+1-\frac{c_0}{\sqrt{\log_2(m)}}}$ maximal pieces. 
	Since $\frac{1}{2\alpha}n\leq m\leq \frac{1}{\alpha}n$ for $n\geq \alpha$, the new constants $\beta$ and $c$ follow from a simple calculation.
\end{proof}

\section*{Acknowledgments}
I would like to thank Henning Bruhn-Fujimoto for helpful discussions.

\printbibliography

\end{document}